\documentclass[11pt,reqno]{amsart}

\usepackage{mathrsfs,amsmath,amssymb,amsfonts,enumerate}

\usepackage[pdftex]{hyperref}

\usepackage[all,2cell]{xy} \UseAllTwocells \SilentMatrices

\begin{document}
\baselineskip=15.5pt

\newtheorem{thm}{Theorem}[section]
\newtheorem{lem}[thm]{Lemma}
\newtheorem{cor}[thm]{Corollary}
\newtheorem{prp}[thm]{Proposition}

\theoremstyle{definition}
\newtheorem{dfn}[thm]{Definition}
\newtheorem{ex}[thm]{Example}
\newtheorem{rmk}[thm]{Remark}

\newcommand{\aro}{\longrightarrow}
\newcommand{\arou}[1]{\stackrel{#1}{\longrightarrow}}

\newcommand{\mm}[1]{\mathrm{#1}}
\newcommand{\bb}[1]{\mathbf{#1}}
\newcommand{\cc}[1]{\mathscr{#1}}

\newcommand{\ce}{\cc{E}}
\newcommand{\cm}{\cc{M}}
\newcommand{\cn}{\cc{N}}
\newcommand{\co}{\cc{O}}
\newcommand{\cq}{\cc{Q}}
\newcommand{\cu}{\cc{U}}
\newcommand{\cv}{\cc{V}}
\newcommand{\cp}{\cc{P}}
\newcommand{\ca}{\cc A}

\newcommand{\g}[1]{\mathfrak{#1}}
\newcommand{\PP}{\mathbf{P}}
\newcommand{\GL}{\mathrm{GL}}
\newcommand{\NN}{\mathbf{N}}
\newcommand{\ZZ}{\mathbf{Z}}
\newcommand{\QQ}{\mathbf{Q}}

\newcommand{\al}{\alpha}

\newcommand{\be}{\beta}

\newcommand{\ga}{\gamma}
\newcommand{\Ga}{\Gamma}

\newcommand{\Om}{\Omega}

\newcommand{\te}{\theta}
\newcommand{\Te}{\Theta}
\newcommand{\ph}{\varphi}
\newcommand{\Ph}{\Phi}
\newcommand{\ps}{\psi}
\newcommand{\Ps}{\Psi}
\newcommand{\ep}{\varepsilon}
\newcommand{\ka}{\kappa}
\newcommand{\si}{\sigma}
\newcommand{\de}{\delta}
\newcommand{\De}{\Delta}
\newcommand{\ze}{\zeta}

\newcommand{\fr}[2]{\frac{#1}{#2}}
\newcommand{\na}{\nabla}
\newcommand{\po}{\cdot}
\newcommand{\hh}[3]{\mathrm{Hom}_{#1}(#2,#3)}
\newcommand{\modules}[1]{#1\text{-}\mathbf{mod}}
\newcommand{\alg}[1]{#1\text{-}\mathbf{alg}}
\newcommand{\spc}{\mathrm{Spec}\,}

\newcommand{\id}{\mathrm{id}}
\newcommand{\ti}{\times}
\newcommand{\tiu}[1]{\underset{#1}{\times}}

\newcommand{\ot}{\otimes}
\newcommand{\otu}[1]{\underset{#1}{\otimes}}

\newcommand{\wt}{\widetilde}
\newcommand{\ov}[1]{\overline{#1}}
\newcommand{\op}{\oplus}
\newcommand{\mc}{\bb{MC}}
\newcommand{\mic}{\mathbf{MIC}}
\newcommand{\mictr}{\mathbf{MIC}^{\rm tr}}
\newcommand{\mctr}{\mathbf{MC}^{\rm tr}}

\title[Connections on trivial bundles]{Connections on trivial vector bundles over projective schemes}

\author[I. Biswas]{Indranil Biswas}

\address{School of Mathematics, Tata Institute of Fundamental
Research, Homi Bhabha Road, Mumbai 400005, India}

\email{indranil@math.tifr.res.in}

\author[P. H. Hai]{Ph\`ung H\^o Hai}

\address{Institute of Mathematics, Vietnam Academy of Science and Technology, Hanoi, 
Vietnam}

\email{phung@math.ac.vn}

\author[J. P. dos Santos]{Jo\~ao Pedro dos Santos}

\address{Institut de Math\'ematiques de Jussieu -- Paris Rive Gauche, 4 place Jussieu,
Case 247, 75252 Paris Cedex 5, France}

\email{joao\_pedro.dos\_santos@yahoo.com}



\date{January 20, 2023.}

\maketitle
\begin{abstract}

Over a smooth and proper complex scheme, the differential Galois group of an integrable connection may be obtained 
as the closure of the transcendental monodromy representation. In this paper, we employ a completely algebraic 
variation of this idea by restricting attention to connections on trivial vector bundles and replacing the 
fundamental group by a certain Lie algebra constructed from the regular forms. In more detail, we show that the 
differential Galois group is a certain ``closure'' of the aforementioned Lie algebra. 
\end{abstract}

\section{Introduction}

A fundamental result of modern differential Galois theory affirms that, for a proper ambient variety, the differential Galois group might be 
obtained as the Zariski closure of the monodromy group. Our objective here is to make a synthesis of results by other mathematicians and 
use this to throw light on a similar finding in the realm of connections on trivial vector bundles. In this case, the role of the fundamental 
group is played by a certain Lie algebra (see Definition \ref{25.06.2020--2}) and the role of the Zariski closure by the group-envelope (see Definition 
\ref{06.08.2020--4}).

Let us be more precise: consider a field $K$ of characteristic zero, a smooth, geometrically connected and proper 
$K$-scheme $X$, and a $K$-point  $x_0$ of $X$. In the special case $K\,=\,\mathbf C$, it is known, mainly due to GAGA, that the category of 
integrable connections on $X$ is equivalent to the category of complex representations of the transcendental 
object $\pi_1(X(\mathbf C),\,x_0)$. In addition, for any such connection $(\ce,\,\na)$, the differential Galois 
group at the point $x_0$ (Definition \ref{25.05.2021--1}) is the Zariski closure of the image ${\rm Im}(M_\ce)$, where 
$M_\ce\,:\,\pi_1(X(\mathbf C),\,x_0)\,\longrightarrow\, {\bf GL}(\ce|_{x_0})$ is the transcendental monodromy 
representation. 

In this work, we wish to draw attention to the fact that the category of integrable connections $(\ce,\,\na)$ on 
trivial vector bundles (that is, $\ce\,\simeq\,\co^{\oplus r}_X$) is  equivalent 
 to the category of representations of {\it a  Lie algebra $\g L_X$}. Then, in the same spirit as the previous paragraph, the differential Galois group of 
$(\ce,\,\na)$ at the point $x_0$ will be the ``closure of the image of $\g L_X$'' in $\bb{GL}(\ce|_{x_0})$ (see 
Definition \ref{06.08.2020--4}). The advantage here is that, contrary to what happens with the computation of the 
monodromy representation in the case $K\,=\,\mathbf C$, the image of $\g L_X$ is immediately visible. See Theorem 
\ref{25.06.2020--1}.
 
Once the above results have been put up, it is very simple to construct connections on curves with prescribed 
differential Galois groups. For this goal, we make use of the fact that semi-simple, respectively reductive, Lie   algebras can be generated by 
solely two elements, respectively three (if the base field is sufficiently  ``large''). See Corollary \ref{14.06.2021--1} and Corollary \ref{04.01.2023--1}.

\subsection*{Acknowledgements}
The first named author is partially supported by a J. C. Bose Fellowship. 
The second named author is supported by the Vietnam Academy of Science and Technology under grant number CTTH00.02/23-24 and the Vingroup Joint Stock Company -- Vingroup Innovation Foundation under the project code VINIF.2021.DA00030.
The third named author thanks M. van der Put for his conscientious reply concerning
the main theme of this work, K.-O. St\"ohr for insights  on the theory of curves, and C. Bonaf\'e and T. Delcroix for help with groups of type $G_2$. He also thanks the
CNRS for granting him a research leave in 2020 which allowed for parts of this text to be elaborated. 
All authors thank the referee for this journal for many remarks which made the text much clearer and for suggesting that we think about what is now Corollary \ref{04.01.2023--1}. 
 
\subsection*{Some notation and conventions}

In all that follows, $K$ is a field of characteristic zero. Vector spaces, associative algebras,
Lie algebras, Hopf algebras, etc, are always to be considered over $K$.

\begin{enumerate}[(1)]
\item The category of finite dimensional vector spaces (over $K$) is denoted by $\bb{vect}$. 

\item The category of Lie algebras is denoted by $\bb{LA}$. The category of Hopf algebras \cite[p.~71]{sweedler69} is denoted by $\mathbf{Hpf}$. 

\item All group schemes are to be affine; $\bb{GS}$ is the category of affine group schemes. Given $G\,\in\,\bb{GS}$, we 
let $\bb{Rep}\,G$ stand for the category of finite dimensional representations of $G$.

\item If $\g A$ stands for an associative algebra, we let $\modules{\g A}$ be the category of left $\g A$-modules which are of finite dimension over $K$. The same notation is invoked for Lie algebras. 

\item An {\it ideal} of an associative algebra is, unless otherwise specified, a two-sided ideal. The {\it tensor algebra} on a vector space $V$ is denoted by $\bb T(V)$. The {\it free algebra} on a set $\{s_i\}_{i\in I}$ is denoted by $K\{s_i\}$.

\item A {\it curve} $C$ is a one dimensional, integral and smooth $K$-scheme. 

\item A {\it vector bundle} is       a locally free   sheaf
of finite rank. A {\it trivial} vector bundle on $X$ is a direct sum of a finite number of copies of $\co_X$. 
\end{enumerate}

\section{Construction of a Hopf algebra}

Let $\Ph$ and $\Ps$ be two finite dimensional vector spaces, and let 
\[
\be\,\,:\,\, \Ph\ot \Ph \,\aro\, \Ps
\] 
be a $K$-linear map with transpose $\be^*\,:\,\Ps^*\,\aro\, \Ph^*\ot \Ph^*$. Let
\[
\g I_\be \,=\, \text{Ideal in $\bb T( \Ph^*)$ generated by}\ {\rm Im}\,\be^* ,
\] 
and define 
\begin{equation}\label{gq}
\g A_{\be}\,=\,\bb T(\Ph^*)/\g I_\be. 
\end{equation}
It is useful at this point to note that $\g I_\be$ is a homogeneous ideal so that $\g A_\be$ has a natural grading.
In more explicit terms, fix a basis $\{\ph_i \}_{i=1}^{r}$ of $\Ph$ and
a basis $\{\ps_i\}_{i=1}^{s}$ of $\Ps$. Write $\{\ph_i^*\}_{i=1}^r$ and $\{\ps_i^*\}_{i=1}^s$ for the respective
dual bases. If
\[
\be(\ph_k \ot \ph_\ell) \,=\, \sum_{i=1}^{s} \be_i^{(k\ell)}\po \ps_i,
\]
then
\[
\be^*( \ps_i^*)\,=\,\sum_{1\le k,\ell\le r} \be_i^{(k\ell)} \po \ph_k^*\ot\ph_\ell^*.
\]
Consequently, $\g A_\be$ in \eqref{gq} is the quotient of the free algebra $K\{t_1,\,\ldots,\,t_{r}\}$ (identified with $\bb T(\Ph^*)$ in the obvious way) by the ideal generated by the $s$ elements
\[
\sum_{1\le k,\ell\le r} \be_i^{(k\ell)} t_kt_\ell,\qquad i\,=\,1,\,\ldots,\,s.
\]
In particular, given $V\,\in\,\bb{vect}$ and elements $A_1,\,\ldots,\,A_{r}\,\in\,\mm{End}(V)$, the association
$t_i\,\longmapsto\, A_i$ defines a representation of $\g A_\be$ if and only if 
\[
\sum_{1\le k,\ell\le r} \be_i^{(k\ell)}\po A_kA_\ell\,=\,0,\qquad\forall\ \,i\,=\,1,\,\ldots,\,s.
\]
It is worth pointing out that if $\be$ is alternating, then 
\begin{equation}\label{01.07.2020--3}
\sum_{1\le k,\ell\le r} \be_i^{(k\ell)} t_kt_\ell\,=\,\sum_{1\le k<\ell\le r} \be_i^{(k\ell)}[t_k,\,t_\ell].
\end{equation}
This reformulation has useful consequences for the structure of $\g A_\be$.

\smallskip
\begin{center}
{\it From now on, $\be$ is always assumed to be alternating}.
\end{center}
\smallskip

Let $\bb L(\Ph^*)$ be the free Lie algebra on the vector space $\Ph^*$ so that $\bb T(\Ph^*)$ is the universal enveloping algebra of $\bb L(\Ph^*)$ \cite[II.3.1, p.~32, Theorem 1]{bourbaki_algebres_lie}. Clearly, abbreviating $\ph_i^*$ to $t_i$, 
\[
\sum_{1\le k<\ell\le r} \be_i^{(k\ell)}[t_k,\,t_\ell]\,\in\, \bb L(\Ph^*),\qquad\forall\ i\,=\,1,\,\ldots,\,s.
\]
Let 
\[
\g K_\be\,=\,\,\,\begin{array}{c}\text{Lie ideal of $\bb L(\Ph^*)$ generated by the $s$}\\ \text{elements 
$\{\sum_{1\le k<\ell\le r} \be_i^{(k\ell)}[t_k,\,t_\ell]\}_{i=1}^s$ in \eqref{01.07.2020--3}}.
\end{array}
\]

\begin{prp}\label{12.01.2021--1} 
The algebra $\g A_\be$ in \eqref{gq} is the universal enveloping algebra of the Lie algebra $\bb L(\Ph^*)/\g K_\be$.
\end{prp}

\begin{proof}This is a consequence of the following general observations. 
Let $\g g$ be a Lie algebra and $\iota:\g g\to U$ be the morphism into its universal enveloping algebra. Let $S\subset\g g$ be a subset and let $S_{\rm Lie}\subset\g g$, respectively $S_{\rm alg}\subset U$,  be the Lie ideal  generated by $S$, respectively the ideal generated by $\iota (S)$.  A moment's thought proves that  $S_{\rm alg}$ is the ideal of $U$ generated by $\iota(S_{\rm Lie})$. 
We conclude that $U/S_{\rm alg}$ is the universal enveloping algebra of the Lie algebra $\g g/S_{\rm Lie}$ \cite[I.2.3, Proposition 3]{bourbaki_algebres_lie}.
Applying this to the Lie ideal $\g K_\be$, the ideal $\g I_\be$ and the   Lie algebra  $\bb L(\Ph^*)$, we arrive at a proof of the proposition.  
\end{proof}

\begin{dfn}The Lie algebra $\bb L(\Ph^*)/\g K_\be$ shall be denoted by $\g L_\be$.
\end{dfn} 

A simple remark should be recorded here. 
\begin{lem}\label{01.07.2020--1}
The above Lie algebra $\g L_\be$ is a quotient of the free Lie algebra $ \bb L(\Ph^*)$. In particular, $\g L_\be$
is generated by the image of $\Ph^*$.
\end{lem}

Recall that for a Lie algebra $L$, the universal enveloping algebra $\bb UL$ has a natural structure of
Hopf algebra \cite[3.2.2, p.~58]{sweedler69} and hence from Proposition \ref{12.01.2021--1} it
follows that $\g A_\be$ has the structure of a Hopf
algebra. Similarly, $\bb T(\Ph^*)$ is also a Hopf algebra and the quotient map
\begin{equation}\label{eq}
\bb T(\Ph^*)\,\longrightarrow\, \g A_\be
\end{equation}
is an arrow of Hopf algebras. 

In what follows, we give the category $\modules{\g A_\be}$ the tensor product explained in \cite[1.8.1, 
p.~14]{montgomery93}. To wit, if $V$ and $W$ are $\g A_\be$--modules, then $V\ot_KW$ is an $\g A_\be$--module by means of the composition $\textcircled{\tiny1}$ below:
\[\xymatrix{
\g A_\be\ot V\ot W\ar@/_1.8pc/[rrrrd]_-{\textcircled{\tiny1}}\ar[rrr]^-{\text{comult.}\ot\id\ot\id}&&&\g A_\be\ot\g A_\be\ot V\ot W\ar@{=}[r]&(\g A_\be  \ot V)\ot (\g A_\be\ot W)\ar[d]^{\text{mult.}\ot\text{mult.}}
\\
&&&&V\ot W.
}\]
It turns out that the canonical equivalence
\begin{equation}\label{06.08.2020--1}
\modules {\g L_\be}\,\,\arou \sim \,\,\modules{\g A_\be}
\end{equation}
is actually a tensor equivalence. 

The only case in which $\g A_\be$ will interest us is that of:

\begin{dfn}\label{25.06.2020--2}
Let $X$ be a smooth, connected and projective $K$-scheme. Let $$\be\,\,:\,\,H^0(X,\, \Om_X^1)\ot H^0(X,\, \Om_X^1)
\,\longrightarrow\, H^0(X,\, \Om_X^2)$$
be the wedge product of differential forms. We put \[\g A_\be\,=\,\g A_X\quad\text{ and }\quad \g L_\be\,=\,\g L_X.\] 
\end{dfn} 
 
\section{Connections}

We shall begin this section by establishing the notation and pointing out  references. We fix a smooth and 
connected $K$-scheme $X$. Soon, we shall assume $X$ to be projective.

\begin{dfn}
We let $\mc$ be the category of $K$--linear connections on coherent $\mathscr O_X$--modules and $\mic$ the full 
subcategory of $\mc$ whose objects are integrable connections \cite[1.0]{katz70}. We let $\mctr$ be the full subcategory 
of $\mc$ having as objects pairs $(\ce,\,\nabla)$ in which $\ce$ is a trivial vector bundle. The category $\mictr$ 
is defined analogously: it is the full subcategory
of $\mic$ having as objects pairs $(\ce,\,\nabla)$ in which $\ce$ is a trivial vector bundle.
\end{dfn}

A fundamental result of the theory of connections is that for each $(\ce,\,\na)$, the coherent sheaf $\ce$ is
actually locally free \cite[Proposition 8.8]{katz70}. Using this and the reconstruction theorem of Tannakian categories \cite[Theorem 2.11]{deligne-milne82}, it is possible to show that, given $x_0\,\in\, X(K)$, the functor
``taking the fibre at $x_0$'' defines a $K$-linear tensor equivalence
\begin{equation}\label{01.07.2020--2}
\bullet|_{x_0}\,\,:\,\, \mic\,\,\arou\sim \,\,\bb{Rep} \, \Pi(X,\,x_0) ,
\end{equation}
where $\Pi(X,\,x_0)$ is a group scheme over $K$. This group scheme is
sometimes called the ``differential fundamental group scheme of $X$ at $x_0$''.
It is in rare cases that $\Pi(X,\,x_0)$ will be an algebraic group (if, for example, $K=\mathbf C$,   $X$ is proper and $\pi_1(X^{\mathrm an})$ is finite), and hence it is important to turn it into a
splice of smaller pieces. This motivates the following definition.

\begin{dfn}[The differential Galois group]\label{25.05.2021--1}Let $(\mathscr E,\,\na)\,\in\,\mic$ be given, and let $\rho_{\mathscr E}\,:\,\Pi(X,\,x_0)\,
\longrightarrow\,\bb{GL}(\mathscr E|_{x_0})$ be the representation associated to $\ce$ via the equivalence
in \eqref{01.07.2020--2}. The image of $\rho_{\mathscr E}$ in $\bb{GL}(\mathscr E|_{x_0})$ is the differential Galois group of
$(\mathscr E,\,\na)$ at the point $x_0$.
\end{dfn} 

\begin{rmk}\label{10.06.2021--2}For $(\ce,\,\na)\,\in\,\bb{MIC}$, the category of representations of the differential
Galois group of $(\ce,\,\na)$ at $x_0$ is naturally a full subcategory of $\bb{MIC}$. For each vector bundle, let us agree to denote by $\check{\mathscr{F}}$
its dual $\mathscr{H}\!om(\mathscr{F},\co_X)$, and endow it with the canonical connection \cite[1.1]{katz70}. It is then  not difficult to see that     
\[
\langle(\ce,\,\na)\rangle_\ot\,\,=\,\left\{ \cm'/\cm''\in\bb{MIC} \,:\,
\begin{array}{c}
\text{there exist $a_i,b_i\in\NN$ such that}
\\
 \text{$\cm''\subset\cm' \subset \bigoplus_i \ce^{\ot a_i}\ot\check\ce^{\ot b_i} $} 
 \end{array}
 \right \}.
\]
See \cite[3.4 and  3.5]{waterhouse79}.
\end{rmk}

From now on, $X$ is in addition {\it projective}. 
Let us be more explicit about objects in $\mctr$. Fix $E\in\bb{vect}$ and let
\[\begin{split}
A&\,\in\,\hh{ \alg K}{\bb T(H^0(X,\, \Om_X^1)^*)}{\,\,\mm{End}(E)}
\\&=
\hh{}{H^0(X,\, \Om_X^1)^*}{\,\,\mm{End}(E)}
\\
&=\mm{End}(E)\ot H^0(X,\, \Om_X^1).
\end{split}\] (We note that in order to make the final identification above, we relied on the fact that $\dim H^0(\Om_X^1)<\infty$.)
Hence, $A$ gives rise to an $\mm{End}(E)$--valued 1--form on $X$ which, in turn, gives rise to a connection 
\begin{equation}\label{eda}
d_A\,:\,\co_X\ot E\,\aro\, (\co_X\ot E)\otu{\co_X}\Om_X^1
\end{equation}
on the trivial vector bundle $\co_X\ot E$.
Explicitly, let $\{\te_i\}_{i=1}^g$ be a basis of $H^0(X,\,\Om_X^1)$ with dual basis $\{\ph_i\}_{i=1}^g$  and let  $A_i\,:=\,A(\ph_i)\, \in\, \text{End}(E)$; we arrive at 
\[
d_{A}(1\ot e) \,=\, \sum_{i=1}^g(1\ot A_i(e))\ot\te_i
\]
for all $e\, \in\, E$.

\begin{dfn}
The above pair consisting of $(\co_X\ot E,\,d_A)$ shall be denoted by $\cv (E,\,A)$.
\end{dfn}

Let now $\{\si_i\}_{i=1}^h$ be a basis of $H^0(X,\, \Om_X^2)$ and write 
\[
\te_k\wedge\te_\ell\,=\,\sum_{i=1}^h\be_{i}^{(k\ell)}\po\si_i.
\]
Since $X$ is proper, Hodge theory tells us that all global 1-forms are closed \cite[Theorem 5.5]{deligne68} and hence
the curvature
\[
R_{d_A}\,:\,\co_X\ot E\,\aro\, \left(\co_X\ot E\right)\ot_{\co_X}\Om_X^2\]
of the connection $d_A$ in \eqref{eda} satisfies
\[
R_{d_A}(1\ot e)\,=\,\sum_{i=1}^h\sum_{1\le k,\ell\le g}\left(1\ot\be_{i}^{(k\ell)}A_kA_\ell(e)\right)\ot\si_i. 
\]
Hence, $R_{d_A}\,=\,0$ if and only if 
\[
\sum_{1\le k,\ell\le g}\be_i^{(k\ell)}A_kA_\ell\,=\,0
\]
for each $i\,\in\,\{1,\,\ldots,\,h\}$.
Also, since $\be$ in Definition \ref{25.06.2020--2} is alternating, we conclude that
$R_{d_A}\,=\,0$ if and only if for each $i\,\in\,\{1,\,\ldots,\,h\}$,
\[
\sum_{1\le k,\ell\le g}\be_i^{(k\ell)}A_kA_\ell\,=\,\sum_{1\le k<\ell\le g}\be_i^{(k\ell)}[A_k,\,A_\ell]\,=\,0.
\]

These considerations form the main points of the proof of the following result, whose thorough verification is 
left to the interested reader. (It is worth noting that $K\,=\,H^0(X,\,\co_X)$ since $X$ is proper, integral and has a $K$-point.)

\begin{prp}The functor 
\[
\cv\,\,:\,\, \modules{\bb T(H^0(\Om_X^1)^*)}\,\aro\, \mctr
\]
is an equivalence of $K$-linear categories. Under this equivalence, $\cv (E,\,A)$ lies in $\mictr$ if and only if $(E,\,A)$ is
in fact a representation of $\g A_X$ (see Definition \ref{25.06.2020--2}).\qed
\end{prp}

Let us now discuss tensor products. Given representations
$$A\,:\,\bb T(H^0(\Om_X^1)^*) \,\longrightarrow\, \mm{End}(E)
\ \text{ and }\ B\,:\,\bb T(H^0(\Om_X^1)^*)\,\longrightarrow\,\mm{End}(F),$$
we obtain a new representation   $A\boxtimes B\,:\,\bb T(H^0(\Om_X^1)^*)\,\longrightarrow\,\mm{End}(E\ot F)$ by putting
\[A\boxtimes B(\ph)\,=\,A(\ph)\ot\id_F+\id_E\ot B(\ph),\qquad\forall\ \ \ph\,\in\, H^0(\Om_X^1)^*.\]
This is of course only the tensor structure on the category $\modules {\bb T(H^0(\Om_X^1)^*)}$ defined by the Hopf algebra structure of $\bb 
T(H^0(\Om_X^1)^*)$ \cite[p.~58]{sweedler69}.
With this, it is not hard to see that the canonical isomorphism of $\co_X$-modules 
\[
\co_X\ot(E\ot F)\,\arou\sim\, (\co_X\ot E)\ot_{\co_X}(\co_X\ot F)
\]
is horizontal with respect to the tensor product connection on the right \cite[Section 1.1]{katz70} and the
connection $d_{A\boxtimes B}$ on the left (it is the connection induced
by the connections $d_A$ and $d_B$). We then arrive at equivalences of tensor categories: 

\begin{thm}\label{06.08.2020--3}\mbox{}
\begin{enumerate}[(i)]
\item The functor 
\[
\cv\,:\, \modules{\bb T(H^0(\Om_X^1)^*)} \,\aro\, \mctr
\]
is an equivalence of $K$-linear tensor categories.

\item The restriction 
\[ \cv\,:\, \modules {\g A_X}\,\aro\,\mictr\]
is also an equivalence of $K$-linear tensor categories. In addition, the composition 
$(\bullet|_{x_0})\circ\cv$ is naturally isomorphic to the forgetful functor, where
$\bullet|_{x_0}$ is constructed in \eqref{01.07.2020--2}.

\item The composition of the equivalence in \eqref{06.08.2020--1}
with $\cv\,:\, \modules {\g A_X}\,\aro\,\mictr$ defines a $K$-linear tensor equivalence 
\[
\modules{\g L_X}\,\arou\sim \,\mictr
\]
(see Definition \ref{25.06.2020--2} for ${\g L_X}$).
\end{enumerate}\qed
\end{thm}

Making use again of the main theorem of (categorical) Tannakian theory, \cite[p.~130, Theorem 2.11]{deligne-milne82},
we obtain an equivalence of $K$-linear tensor categories: 
\begin{equation}\label{06.08.2020--2}
\bullet|_{x_0} \,\,:\,\, \mictr\,\arou\sim \,\bb{Rep}\,\Te(X,\,x_0) , 
\end{equation}
where $\Te(X,\,x_0)$ is a group scheme.
In addition, the inclusion map $$\mictr\,\longrightarrow\, \mic$$ defines a morphism 
\[
\bb q_X\,:\,\Pi(X,\,x_0)\,\aro\,\Te(X,\,x_0),
\]
where $\Pi(X,\,x_0)$ and $\Te(X,\,x_0)$ are constructed in \eqref{01.07.2020--2} and
\eqref{06.08.2020--2} respectively.
Along the lines of Proposition 3.1 of \cite{biswas-et-al21}, we have: 

\begin{prp}\label{30.06.2020--2}
The above morphism $\mathbf q_X$ is in fact a quotient morphism. 
\end{prp}

\begin{proof}
Let $\mathscr E\,\aro\,\mathscr Q$ be an epimorphism of $\mic$ with $\mathscr E\,\in\,\mictr$; write $e$ for the rank of $\ce$ and $q$
for that of $\mathscr Q$.

Let $G$ stand for the Grassmannian variety of $q$--dimensional quotients of $K^{\oplus e}$, and let $\co_G^{\oplus e}
\,\aro\,\cu$ stand for the universal epimorphism \cite[5.1.6]{nitsure05}. We then obtain a morphism $f\,:\,X\,\aro
\, G$ such that $f^*\cu\,=\,\cq$. For each projective curve $\ga\,:\,C\,\aro\, X$, the vector bundle $\ga^*\cq\,=\
(f\circ\ga)^*\cu$, which carries a connection, has degree zero \cite[Remark 3.3]{biswas-subramanian06}; in particular, $(f\circ\ga)^*\det\cu$ has
also degree zero. As $\det\cu$ is a very ample invertible sheaf on $G$ \cite[p. 114]{nitsure05}, from $\text{degree}((f\circ\ga)^*\det\cu)\,=\, 
\text{degree}((f\circ\ga)^*\cu)\,=\,0$ we conclude that $(f\circ\ga)^*\det\cu$ is trivial, and hence 
the schematic image of $f\circ\ga$ is a (closed) point of $G$ \cite[p.~331, Exercise 8.1.7(a)]{liu02}. Now, Ramanujam's 
Lemma (see Remark \ref{ramanujam_lemma} below) can be applied to show that any two closed points $x_1$ and $x_2$ of 
$X$ belong to the image of a morphism $\ga\,:\,C \,\aro\, X$ from a projective curve.
Therefore, the schematic image of $X$ under $f$ is a single point (necessarily closed) and hence $f$ factors as $X\,\aro\,\spc K'\,\aro\, G$, with $K'$ a finite extension of $K$. Since $H^0(X,\,\co_X)\,=\,K$, it  must be the case that $K'=K$,  and hence $f$ factors through the structural morphism $X\,\aro\,\spc K$. Consequently, $f^*\cu
\,=\, \mathscr Q$ is a trivial vector bundle. The standard criterion for a morphism of group schemes
to be a quotient morphism (see \cite[p.~139, Proposition 2.21(a)]{deligne-milne82} for the
criterion) can be applied to complete the proof.
\end{proof}

\begin{rmk}[{Ramanujam's Lemma}]\label{ramanujam_lemma}
Let $Z$ be a geometrically integral projective $K$-scheme and $z_1,\,z_2$   two closed points on it. 
We contend that there exists a projective curve $C$ together with a morphism $\ga\,:\,C
\,\longrightarrow\, X$ such that $z_1$ and $z_2$ belong to the image of $\ga$. The proof is the same
as in \cite[p.~56]{mumford70}, but the Bertini theorem necessary for our
purpose comes from \cite[Cor. 6.11]{jouanolou83}. 

If $\dim Z\,=\,1$, it is sufficient to chose $C$ to be the normalisation of $Z$. Let $\dim Z\,:=\,d\ge2$ and 
suppose that the result holds for all geometrically integral and projective schemes of dimension strictly smaller 
than $d$. We only need now to find a geometrically integral closed subscheme $Y\,\subset\, Z$ containing $z_1$ and 
$z_2$ and having dimension strictly smaller than $d$. Let $\pi\,:\,Z'\,\longrightarrow\, Z$ be the blow up of
the closed subscheme $\{z_1,\,z_2\}$. Note that $Z'$ is geometrically
integral \cite[8.1.12(c) and (d), p.~322]{liu02}. In addition, the 
fibres of $\pi$ above $z_1$ and $z_2$ are Cartier subschemes of $Z'$ and hence of dimension at least 1
\cite[2.5.26, p.~74]{liu02}. Let $Z'\,\longrightarrow\,\PP^N$ be a closed immersion, and
let $H\,\subset\,\PP^N$ be a hyperplane       
such that
\begin{itemize}
\item $Z'\cap H$ is geometrically integral \cite[6.11 (2)-(3)]{jouanolou83} of dimension 
$\dim Z-1$, and

\item $Z'\cap\pi^{-1}(z_i)\,\not=\,\emptyset$, loc. cit, (1)-(b).
\end{itemize}
Then, the schematic 
image of $\pi\,:\,Z'\cap H\,\longrightarrow\, Z$ is the scheme $Y$ that we are seeking.
\end{rmk}

\begin{rmk}In \cite[Ch. 8, p.~331, Exercise 1.5]{liu02}, the reader shall find a useful, but slightly
weaker version of Ramanujam's Lemma. 
\end{rmk}

\begin{rmk}The idea to consider certain connections as representations of a Lie algebra can be found at least on 
\cite[12.2--5]{deligne87}.
\end{rmk}

\section{The Tannakian envelope of a Lie algebra}

All that follows in this section is, in essence, due to Hochschild \cite{hochschild59}; since he expressed himself without using group 
schemes and his ideas are spread out in several papers, we shall briefly condense his theory in what follows. The reader should also consult 
\cite{nahlus02}, where some results reviewed here also appear. 

Our objective in this section is to give a construction of the affine envelope of a Lie algebra. One can, of course, 
employ the categorical Tannakian theory \cite[p.~130, Theorem 2.11]{deligne-milne82} to the category $\modules {\g 
L}$ to obtain such a construction, but we prefer to draw the reader's attention to something which is less 
widespread than \cite{deligne-milne82} and more concrete.

Let $\g L$ be a Lie algebra with universal enveloping algebra $\bb U\g L$. Note that $\bb U\g L$ is not only an 
algebra, but also a cocommutative Hopf algebra; see \cite[p.~58, Section 3.2.2]{sweedler69} and \cite[p.~72, Example 
1.5.4]{montgomery93}. Consequently, the Hopf dual \[(\bb U\g L)^\circ \,=\,\left\{\ph:\bb U\g L\to K\,:\,\begin{array}{c}\text{$\ph$ vanishes on a subspace } \\ \text{ of finite codimension}\end{array}\right\}\]   is a commutative Hopf 
algebra (see \cite[Section 6.2, pp. 122-3]{sweedler69} or \cite[Theorem 9.1.3]{montgomery93}). This means that
\[
{\mathbf G}(\g L)\,:=\,\spc\,(\bb U\g L)^\circ
\]
is a group scheme, which we call the {\it affine envelope of $\g L$}. Let us show that this construction gives a left
adjoint to the functor 
\[
{\rm Lie}\,:\,\bb{GS}\,\aro\, \bb{LA}.
\]

We start by noting that $\bb G$ is indeed a functor; given a morphism of Lie algebras $\g G\,\longrightarrow\, \g H$, the associated 
morphism  $ \bb U\g G \,\longrightarrow\,  \bb U\g H $ gives rise to a map of coalgebras $(\bb U\g 
H)^\circ\,\longrightarrow\, (\bb U\g G)^\circ$; see \cite[p.~114, Remark 1]{sweedler69}. The fact that the algebra 
structures are also preserved is   a consequence of the fact that $\bb U\g G\,\longrightarrow\,\bb U\g H$ is 
also an arrow of coalgebras.

Let $G$ be a group scheme, and let $\rho\,:\,\g L\,\longrightarrow\, {\rm Lie}\, G$ be a morphism of Lie algebras; we write $\rho$ for the map
induced between universal enveloping algebras as well. Interpreting elements in ${\rm Lie}\,G$ as elements of
${\rm End}_K(\co(G))$, we obtain a morphism of $K$-algebras $\bb U({\rm Lie}\,G)\,\longrightarrow\, {\rm End}_K
(\co(G))$. To continue, we recall that a module over an algebra   is   {\it locally finite} if it is the union of finite dimensional submodules. In addition, such a notion is easily adapted to include modules over Lie algebras or comodules over coalgebras. 
Because of \cite[3.3 Theorem]{waterhouse79}, the $\co(G)$-comodule $\co(G)$ is locally finite and hence is also   locally finite as a ${\rm Lie}\,G$--module,  $\g L$--module or     $\bb U\g L$--module. Let
\[
\ph_\rho\,:\, \co(G)\,\aro\, (\bb U\g L)^*
\]
be defined by 
\begin{equation}\label{pra}
\ph_\rho(a)\,:\,u\,\longmapsto\,\ep(\rho(u)(a)),
\end{equation}
where
\begin{equation}\label{pra2}
\ep\,:\,\co(G)\,\longrightarrow\,K
\end{equation}
is the counit and $u\, \in\, \bb U\g L$. 

\begin{lem}\mbox{}
\begin{enumerate} [(1)]
\item For each $a\in \co(G)$, the element $\ph_\rho(a)$ in \eqref{pra} lies in the Hopf dual $(\bb U\g L)^\circ$.

\item The map $\ph_\rho$ is a morphism of Hopf algebras.
\end{enumerate}
\end{lem}

\begin{proof}(1) Because $\co(G)$ is a locally finite $\bb U\g L$--module, as explained above, the element  $a$ belongs to a finite dimensional subspace $V$ which is stable under $\bb U\g L$.  
Let $I\,\subset\,\bb U\g L$ be the kernel of the induced map of $K$-algebras 
$\bb U\g L\,\stackrel\rho\longrightarrow\,\,\bb U(\mm{Lie}\,G)\,\longrightarrow\,\mm{End}(V)$; it follows that $I\,\subset\, \mm{Ker}\,\ph_\rho(a)$ and $\ph_\rho(a)\,\in\,(\bb U\g L)^\circ$. 

(2) This verification is somewhat lengthy, but straightforward once the right path has been found. We shall only indicate the most important ideas. Let us write $\ph$ instead of $\ph_\rho$ and consider elements of $\bb U\g L$ as $G$-invariant linear operators \cite[Section 12.1]{waterhouse79} on $\co(G)$. In what follows, we shall use freely the symbol $\Delta$ to denote comultiplication on different coalgebras. 

{\it Compatibility with multiplication.} We must show that
\begin{equation}\label{ets}
 [\ph(a)\ot\ph (b)]\left(\Delta (u)\right)\,=\,\ph(ab)(u)
\end{equation}
for all $a,\,b\,\in\,\co(G)$ and $u\,\in\,\bb U\g L$. Obviously, eq. \eqref{ets} holds for $u\,\in\, K\,\subset\,
\bb U\g L$. In case $u\,\in\,\g L$, the validity of eq. \eqref{ets} is an easy consequence of the fact that
$u\,:\,\co(G)\,\longrightarrow\,\co(G)$ is a derivation and $\Delta u\,=\,u\ot1+1\ot u$.
We then prove that if eq. \eqref{ets} holds for $u$ then, for any given $\delta\,\in\,\g L$, eq. 
\eqref{ets} holds for $u\delta$. Since $\bb U\g L$ is generated by $\g L$, we are done. 

{\it Compatibility with comultiplication.} For $\zeta\,\in\,(\bb U\g L)^\circ$, we know that $\Delta_{(\bb U\g L)^\circ}
(\zeta)$ is defined by 
\[
u\ot v\,\longmapsto\,\zeta(uv)
\]
for $u,\, v\,\in\,\bb U\g L$.
We need to prove that $$\ep (uv(a))\,=\,(\ph\ot\ph)\circ\Delta a$$ for every triple $u,\, v\,\in\,\bb U\g L$ and
$a\,\in\,\co(G)$, where $\ep$ is the homomorphism in
\eqref{pra2}. This follows from the invariance formulas $\Delta u\,=\,(\id\ot u)\Delta$.

{\it Compatible with unit  and co-unit.} This is much simpler and we omit its verification. 

{\it Compatibility with antipode.} Since $\ph$ respects multiplication and comultiplication, unit and co-unit, it is a morphism of  bialgebras. Now, \cite[Lemma 4.0.4]{sweedler69} guarantees that $\ph$ is compatible with the antipode.
 \end{proof}

\begin{prp}\label{30.06.2020--1}The above construction establishes a bijection
\begin{eqnarray*}
\ph\,\,:\,\,\hh{\bb{LA}}{\g L}{\,\,{\rm Lie}\,G}&\,\aro\,&\hh{\bb{Hpf}}{\co(G)}{\,\,(\bb U\g L)^\circ}\\
&& = \hh{\bb{GS}}{{\mathbf G}(\g L)}{\,\, G}, 
\end{eqnarray*}
rendering $\bb G\,:\,\bb{LA}\,\longrightarrow\,\bb{GS}$ a left adjoint to ${\rm Lie}\,:\,\bb{GS}
\,\longrightarrow\,\bb{LA}$. 
\end{prp}

\begin{proof} 
We construct the inverse of $\ph$ and leave the reader with all verifications.
Let $f\,:\,\co(G)\,\longrightarrow\,(\bb U\g L)^\circ$ be a morphism of Hopf algebras. Let $x\,\in\,\g L$ be given, and define
\begin{equation}\label{08.06.2021--1}
\psi_f(x)\,:\, \co(G)\aro K,\quad a\,\longmapsto\, f(a)( x).
\end{equation}
It is a simple matter to show that $\psi_f(x)$ is an $\varepsilon$-derivation (see \cite[12.2]{waterhouse79} for the definition),  which is then interpreted as an 
element of ${\rm Lie}\, G$ in a standard fashion. In addition, $\psi_f\,:\,\g 
L\,\longrightarrow\,{\rm Lie}\, G$ gives a morphism of Lie algebras (the reader might use the bracket as explained 
in \cite[Section 12.1, p.~93]{waterhouse79}). Then $f\,\longmapsto\,\ps_f$ and $\rho\,\longmapsto\,\ph_\rho$ are 
mutually inverses; the verification of this fact consists of a chain of simple manipulations and we contend 
ourselves in giving some elements of the equations to be verified. That $\ps_{\ph_\rho}\,=\,\rho$ is
in fact immediate. On the 
other hand, the verification of
\[\ph_{\ps_f}(a)(u)\,=\,f(a)(u),\qquad\forall\ \,a\,\in \,\co(G),\,\forall\ \,\,u\,\in\,\bb U\g L\] 
requires the ensuing observations. (We shall employ Sweedler's notation for the Hopf algebra $\co(G)$ \cite[Section 
1.2, 10ff]{sweedler69}.)
\begin{enumerate}[(1)]
\item For $\delta\,\in\,\g L$, the derivation $\co(G)\,\longrightarrow\,\co(G)$ associated to $\ps_f(\delta)$ is
determined by $a\,\longmapsto\, \sum_{(a)} a_{(1)}\po [f(a_{(2)})(\delta)]$.
\item The axioms for the coproduct and co-unit show that $\sum_{(a)}\ep(a_{(1)})a_{(2)}\,=\,a$. 
\item Suppose that for $u\,\in\,\bb U\g L$ and $\de\,\in\,\bb U\g L$ we know that, for all $a\,\in\,\co(G)$,
\[
\ph_{\ps_f}(a)(u)\,=\, f(a)(u)\quad\text{ and }\quad \ph_{\ps_f}(a)(\delta)\,=\, f(a)(\delta). 
\]
Then $\ph_{\ps_f}(a)(u\delta)\,=\,f(a)(u\delta)$ because of the equations 
\[
f(a)(xy)\,=\,\sum_{(a)} f(a_{(1)})(x)\po f(a_{(2)})(y),\qquad\forall\ \,\,x,\,y\,\in\,\bb U\g L,
\]
which are a consequence of the fact that $f$ is a map of Hopf algebras.
\end{enumerate}
This completes the proof.
\end{proof}

In the proof of Proposition \ref{30.06.2020--1} we defined a bijection 
\[
\ps\,\,:\,\,\hh{\bb{GS}}{\bb G(\g L)}{G}\,\aro \,\hh{\bb{LA}}{\g L}{{\rm Lie}\,G} 
\]
by means of eq. \eqref{08.06.2021--1}. 
In case $G\,=\,\bb{GL}(V)$ and
in the light of the identification ${\rm Lie}\,\bb{GL}(V)\,=\,\g{gl}(V)$, $\ps$ has a rather useful description. Let
$f\,:\,\bb G(\g L)\,\longrightarrow\,\bb{GL}(V)$ be a representation and let $c_f:V\to V\ot (\bb U\g L)^\circ$ be the associated comodule morphism. 
It then follows that 
\begin{equation}\label{10.06.2021--1}
(\id_V\ot \text{evaluate at $x$})\circ c_f \,=\, \ps_f(x).
\end{equation}

\begin{cor}\label{07.06.2021--2}Let $V$ be a finite dimensional vector space and
$f\,:\, \bb G(\g L)\,\longrightarrow\,\bb{GL}(V)$ a representation. Write $\ps_f\,: \,\g L\,\longrightarrow\, \g{gl}(V)$
for the morphism of Lie algebras mentioned above. Then, this gives rise to a $K$-linear equivalence of tensor categories \[\bb{Rep}\,\bb{G}(\g L)\aro \modules{\g L} .\]
\end{cor}

\begin{proof}
To define a functor $\bb{Rep}\,\bb{G}(\g L)\,\longrightarrow\, \modules{\g L}$ it is still necessary to 
define the maps between sets of morphisms.

Let $f\,:\,\bb G(\g L)\,\longrightarrow\,\bb{GL}(V)$ and $g\,:\,\bb G(\g L)\,\longrightarrow\,\bb{GL}(W)$ be
representations, and let $T\,\in\,\hh{K}{V}{\,W}$. 
We shall show that $T\,\in\,\hh{\bb G(\g L)}{V}{\,W}$ if and only if $T\,\in\,\hh{\g L}{V}{\,W}$.

Consider $\widehat T\,=\,\begin{pmatrix}I&0\\T&I\end{pmatrix}\,\in \,\GL(V\op W)$ and denote by
$$C_T\,\,:\,\,\bb{GL}(V\op W)\,\longrightarrow\,\bb{GL}(V\op W)$$ the conjugation by $\widehat T$.
Then $T$ is $G$--equivariant if and only if $C_T\circ\begin{pmatrix}f& 0\\ 0&g\end{pmatrix}
\,=\,\begin{pmatrix}f&0\\ 0&g\end{pmatrix}$. Similarly let us write $c_T\,:\,\g{gl}(V\op W)
\,\longrightarrow\,\g {gl}(V\op W)$ to denote conjugation by $\widehat T$. Then, 
for given representations $\rho\,:\,\g L\,\longrightarrow\,\g{gl}(V)$ and $\si\,:\,\g L\,\longrightarrow\,
\g{gl}(W)$, the arrow $T$ is a morphism of $\g L$-modules if and only if $c_T\begin{pmatrix}\rho& 0\\ 0&\si\end{pmatrix}
\,=\,\begin{pmatrix}\rho& 0\\ 0&\si\end{pmatrix}$. Employing equation \eqref{10.06.2021--1}, we verify readily that
\[
\xymatrix{
\hh{}{\bb G(\g L)}{\bb{GL}(V\op W)}\ar[rr]^{C_T\circ(-)}\ar[d]_\ps&& \hh{}{\bb G(\g L)}{\bb{GL}(V\op W)}\ar[d]^\ps
\\
 \hh{}{ \g L }{\g {gl}(V\op W)}\ar[rr]_{c_T\circ(-)}&& \hh{}{ \g L }{\g{gl}(V\op W)}  }
\]
commutes. We then see that $\ps$ becomes a functor, which is $K$-linear, exact and fully-faithful. 

Let us now deal with the tensor product. Given representations $f\,:\,\bb{G}(\g L)\,\longrightarrow\, \bb{GL}(V)$ and
$g\,:\,\bb{G}(\g L)\,\longrightarrow\,\bb{GL}(W)$, let us write $$t:\bb{G}(\g L)
\,\longrightarrow\,\bb{GL}(V\ot W)$$ for the tensor product 
representation. We then obtain on $V\ot W$  the structure of an $\g L$-module via $\ps_{t}$ and it is to 
be shown that this is precisely the $\g L$-module structure coming from the tensor product of $\g L$-modules. In 
other words, we need to show that for any $x\,\in\, \g L$, the equation 
$\ps_f(x)\ot\id_W+\id_V\ot\ps_g(x)\,=\,\ps_{t}(x)$ holds. We make use of eq. \eqref{10.06.2021--1} again. 
Let $v\,\in\, V$ and $w\,\in\, W$ be such that $c_f(v)\,=\,\sum_iv_i\ot f_i$ and $c_g(w)\,=\,\sum_jw_j\ot g_j$. Then 
$c_{t}(v\ot w)\,=\,\sum_{i,j}v_i\ot w_j\ot f_ig_j$ and hence $$\ps_{t}(x)(v\ot w) \,=\, 
\sum_{i,j}v_i\ot w_j\po(f_i(x)\ep(g_j)+\ep(f_i)g_j(x)),$$ where $\ep\,:\,(\bb U\g L)^\circ
\,\longrightarrow\, K$ is the co-unit defined 
by evaluating at $1\,\in\,\bb U\g L$, and we have used that $\De x\,=\,x\ot1+1\ot x$ (which is true, by definition of the coproduct, for all elements of $\g L$ inside $\bb U\g L$). Now, $\sum_iv_i \ep(f_i)\,=\,v$
and $\sum_jw_j \ep(g_j)\,=\,w$. Hence, $\ps_{t}(x)(v\ot w)\,
=\,\sum_if_i(x)v_i\ot w+\sum_jv\ot g_j(x)w_j$, as we wanted.
\end{proof}

\begin{cor}\label{07.06.2021--1}
Let $G$ be an algebraic group scheme, and let ${\mathbf G}(\g L)\,\longrightarrow\,
G$ be a quotient morphism. Then $G$ is connected. Said differently, ${\mathbf G}(\g L)$ is pro-connected. 
\end{cor}

\begin{proof}For the finite \'etale group scheme $\pi_0(G)$ \cite[Section 6.7]{waterhouse79}, the set 
\[
\hh{\bb{LA}}{\g L}{\, {\rm Lie}\,(\pi_0G)} = \hh{\bb{LA}}{\g L}{\,0 }
\]
is a singleton and hence $\hh{\bb{GS}}{{\mathbf G}(\g L)}{\,\,\pi_0(G)}$ is a singleton because of Proposition \ref{30.06.2020--1}. It then follows that
$\pi_0(G)$ is trivial and $G$ is connected \cite[Theorem of 6.6]{waterhouse79}. 
\end{proof}

In what follows, we denote by
\begin{equation}\label{echi}
\chi\,\,:\,\, \g L\,\aro\,{\rm Lie} \,{\mathbf G}(\g L)
\end{equation}
the morphism of Lie algebras corresponding to the identity of $\hh{\bb{GS}}{{\mathbf G}(\g L)}{\,\,{\mathbf G}(\g L)}$ under the bijection in Proposition \ref{30.06.2020--1}. This is, of course, the unit of the adjunction \cite[IV.1]{maclane70}. Let us profit to note that, as explained in \cite[IV.1, eq. (5)]{maclane70}, for each $f\in \hh{\bb{GS}}{\bb G(\g L)}{G}$, the equation 
\begin{equation}\label{01.06.2021--2}
\ps_f\,=\,({\rm Lie}\,f)\circ\chi.
\end{equation}
is valid.

One fundamental property of $\chi$ needs to be expressed in terms of ``algebraic density'' \cite[p.~175]{hochschild74}.

\begin{dfn}Let $G$ be a group scheme. A morphism $\rho\,:\,\g L\,\longrightarrow\,
{\rm Lie}\,G$ is {\it algebraically dense} if the only closed subgroup scheme $H\,\subset\, G$
such that $\rho(\g L)\,\subset\,{\rm Lie}\,H$ is $G$ itself. If $\rho$ happens to be an injection, we shall simply say that $\g L$ is algebraically dense.
\end{dfn}

\begin{prp}\label{01.06.2021--3}
The morphism $\chi\,:\,\g L\,\longrightarrow\, {\rm Lie}\,{\mathbf G}(\g L)$ in \eqref{echi} is algebraically dense. 
\end{prp}

To prove Proposition \ref{01.06.2021--3}, we   require the following fact. (The proofs of the first claims are in \cite[Corollary in 3.3, p.~24]{waterhouse79} and \cite[Theorem of 
14.1, p.~109]{waterhouse79}, while the proof of the final claim can be found in \cite[II.5, Proposition 5.3, p.250]{demazure-gabriel70}.) 

\begin{lem}\label{01.06.2021--4} Let $G$ be a group scheme. Then there exists a projective system of algebraic group schemes
$\{G_i,\,u_{ij}\,:\,G_j\,\longrightarrow\, G_i\}$ where each $u_{ij}$ is faithfully flat and an isomorphism
$G\,\simeq\, \varprojlim_iG_i$. In addition, all arrows ${\rm Lie}\,u_{ij}\,:\,
{\rm Lie}\,G_j\,\longrightarrow\,{\rm Lie}\,G_i$ are surjective.\qed
\end{lem}


\begin{proof}[Proof of Proposition \ref{01.06.2021--3}]
Let $u\,:\,H\,\longrightarrow\, {\mathbf G}(\g L)$ be a closed immersion and let
$\rho\,:\,\g L\,\longrightarrow\, {\rm Lie}\,H$ be a morphism of
Lie algebras such that 
\[({\rm Lie}\,u)\circ\rho\,=\,\chi.
\] 
Let $f\,:\,{\mathbf G}(\g L)
\,\longrightarrow\, H$ be a morphism of group schemes such that $\rho\,=\,\ps_f$. From eq. \eqref{01.06.2021--2}, we have
\[
\rho\,=\,({\rm Lie}\,f)\circ\chi.
\]
Hence, $\chi\,=\,({\rm Lie}(u\circ f))\circ\chi$, which proves that $u\circ f\,=\,\id_{\bb G(\g L)}$ (see eq. \eqref{01.06.2021--2}). In particular, ${\rm Lie}\,u$ is 
surjective. 

Let us now write $$\bb G(\g L)\,=\,\varprojlim_iG_i$$ as in Lemma \ref{01.06.2021--4}. Define $H_i$ as being the image
of $H$ in $G_i$; a moment's thought shows that $$H\,=\,\varprojlim_iH_i ,$$ and that the the transition arrows of
the projective system $\{H_i\}$ are also faithfully flat. This being so, the morphisms between Lie algebras in the
projective system $\{H_i\}$ are all surjective \cite[II.5, Proposition 5.3, p.250]{demazure-gabriel70}.
Consequently, the obvious morphisms ${\rm Lie}\,\bb G(\g L)\,\longrightarrow\, {\rm Lie}\,G_i$ and ${\rm Lie}\,H
\,\longrightarrow\,{\rm Lie}\,H_i$ are always surjective. 
Hence, the natural morphisms  ${\rm Lie}\,H_i\,\longrightarrow\,{\rm Lie}\,G_i$ are always surjective. 

Using \cite[Proposition II.6.2.1, p.~259]{demazure-gabriel70} and the fact that each $G_i$ is connected, we 
conclude that $H_i\,=\,G_i$, and $H\,=\,{\mathbf G}(\g L)$. This proves Proposition \ref{01.06.2021--3}.
\end{proof}

Let $G$ be an algebraic group scheme with Lie algebra $\g g$. Recall that a {\it Lie subalgebra of $\g g$ is algebraic }
if it is the Lie subalgebra of a closed subgroup scheme of $G$ \cite[Definition II.6.2.4]{demazure-gabriel70}. As 
argued in \cite[II.6.2, p.~262]{demazure-gabriel70}, given an arbitrary Lie subalgebra $\g h\,\subset\,\g g$, there 
exists a smallest algebraic Lie subalgebra of $\g g$ containing $\g h$: it is the (algebraic) {\it envelope} of $\g 
h$ inside $\g g$. Allied with \cite[II.6.2.1a, p.~259]{demazure-gabriel70}, it then follows that there exists a 
{\it smallest closed and connected subgroup scheme} of $G$ whose Lie algebra contains $\g h$. This group carries no 
name in \cite{demazure-gabriel70}, so we shall allow ourselves to put forward:

\begin{dfn}\label{06.08.2020--4}
Let $G$ be an algebraic group scheme and $\g h\,\subset\, \mm{Lie}\,G$ a Lie subalgebra. The {\it group-envelope} of $\g h$ is the smallest
closed subgroup scheme of $G$ whose Lie algebra contains $\g h$. We also define the {\it group-envelope} of a subspace $V\subset {\rm Lie}\,G$ as being the group-envelope of the Lie algebra generated by $V$ in ${\rm Lie}\,G$. 
\end{dfn}
 
\begin{thm}[{\cite[Theorem 1, \S~3]{hochschild59}}]\label{25.06.2020--3}
Let $f\,:\,{\mathbf G}(\g L)\,\longrightarrow\, \bb{GL}(E)$ be the representation
associated to the $\g L$--module $\rho\,:\,\g L\,\longrightarrow\,\g{gl}(E)$, that is, $\ps_f\,=\,\rho$. Then
the image $I\,=\, {\rm image}(f)$ of ${\mathbf G}(\g L)$ in $\bb{GL}(E)$ is the group-envelope
of $\rho(\g L)\,\subset\, \g{gl}(E)$.
\end{thm}

\begin{proof}
Consider a factorization $\rho\,:\,\g L\,\longrightarrow\, {\rm Lie}\, H$, where $H\,\subset\,\bb{GL}(E)$ is closed.
Because $\rho\,=\,({\rm Lie}\, f)\circ \chi$ (see eq. \eqref{01.06.2021--2}), it follows that $\chi(\g L)
\,\subset \,({\rm Lie}\,f)^{-1}({\rm Lie}\,H)$. We now observe that  the natural inclusion    
\[{\rm Lie}\,(f^{-1}(H))\subset ({\rm Lie}\,f)^{-1}({\rm Lie}\,H)\] is an {\it equality}, see Lemma \ref{31.12.2022--1} below.  (We are   unable to find a reference for this simple fact.) 
Hence, $f^{-1}(H)\,=\,{\mathbf G}(\g L)$ because $\chi\,:\,\g L
\,\longrightarrow\,{\rm Lie}\,{\mathbf G}(\g L)$ is algebraically dense. This implies that
$I\,\subset\, H$. Because $\rho(\g L)\,=\,{\rm Lie}(f)\circ\chi(\g L)$, we deduce that ${\rm Lie}\,I
\,\supset\, \rho(\g L)$, so that $I$ is the group-envelope. 
\end{proof}

\begin{lem}\label{31.12.2022--1}Let $f:G'\,\aro\, G$ be a morphism of group schemes. Let $H\subset G$ be a closed subgroup. Denote by $H'$ its inverse image in $G'$. Then, ${\rm Lie}\,H'=\mm{Lie}(f)^{-1}(\mm{Lie}\,H)$.\end{lem}

\begin{proof} It is only needed to show the inclusion $\mm{Lie}\,H'\supset(\mm{Lie}\,f)^{-1}(\mm{Lie}\,H)$. We consider elements of Lie algebras as $\ep$-derivations, cf. the proof of Theorem 12.2 in \cite{waterhouse79}. 
Let $\partial:\co(G')\,\aro\,K$ be an $\ep$-derivation whose image in ${\rm Lie}\,G$ is induced by an $\ep$-derivation $\co(H)\,\aro\,K$. Now, if $I\subset\co(G)$ is the ideal of $H$, we conclude that   
\[
\co(G)\,\aro\,\co(G')\stackrel \partial\longrightarrow K\]
annihilates $I$, and hence that $\partial$ annihilates       $I\co(G)$. But this is  the ideal of   $f^{-1}(H)$ and hence $\partial:\co(G')\aro K$ comes from an $\ep$-derivation $\co(H')\aro K$. 
\end{proof}
\section{The differential Galois group}

In this section, $X$ is assumed to be a projective, connected and smooth $K$-scheme and $x_0$ a $K$-point of $X$. We 
recall that $\Te(X,\,x_0)$ is the group scheme constructed in eq. \eqref {06.08.2020--2}.

Using the tensor equivalences 
\[\bb{Rep}\,\bb G(\g L_X)\,\,\arou\sim\,\,
\modules{ \g L_X}\,\,\arou\sim\, \,\modules{\g A_X}\arou{\cv}\mictr\,\,\arou{\bullet|_{x_0}}\,\, \bb{Rep}\,\Te(X,x_0)
\]
obtained by eq. \eqref{06.08.2020--1}, Theorem \ref{06.08.2020--3} and Corollary \ref{07.06.2021--2}, we derive an isomorphism 
\[\ga:
\Te(X,\,x_0)\,\,\arou\sim \,\,\bb G(\g L_X) 
\]
such that the corresponding functor $\ga^\#\,:\,\bb{Rep}\,\bb G(\g L_X)\,\longrightarrow\,
\bb{Rep}\,\Te(X,\,x_0)$ is naturally isomorphic to the above composition.

Let $(E,\,A)\,\in\,\modules{\g A_X}$ be given. With an abuse of notation, we shall let $A$ denote the linear map $H^0(\Om_X^1)^*
\,\longrightarrow\, \mm{End}(E)$, 
the morphism of associative algebras $\g A_X\,\longrightarrow\, \mm{End}(E)$ or the morphism of Lie algebras
$\g L_X \,\longrightarrow\, \mm{End}(E)$.

\begin{thm} \label{25.06.2020--1}The differential Galois group of $\cv(E,\,A)\,=\,(\co_X\ot E,\,d_A)$ is
the group-envelope of $A(H^0(X,\, \Om_X^1)^*)$.

Said otherwise, given a trivial vector bundle $\ce$ or rank $r$ with global basis $\{e_i\}_{i=1}^r$, an integrable connection
$$\na\,:\,\ce\,\longrightarrow\,\ce\ot\Om_{X}^1$$ and a basis $\{\te_j\}_{j=1}^g$ of $H^0(X,\,\Om_X^1)$, define matrices
$A_k\,=\,(a^{(k)}_{ij})_{1\le i,j\le r}\,\in\, {\rm M}_r(K)$ by
\[
\na e_j\,=\,\,\,\sum_{k=1}^g \sum_{i=1}^r a^{(k)}_{ij}\po e_i\ot \te_k.
\]
Then, the differential Galois group of $\ce$ at the point $x_0$ is isomorphic to the group-envelope
in $\bb{GL}_r$ of the Lie algebra generated by $\{A_k\}_{k=1}^g$.
\end{thm}

\begin{proof}
We note that the Lie subalgebra of $\mm{End}(E)$ generated by $A(H^0(\Om_X^1)^*)$ is the image of $A(\g L_X)$; 
indeed, as a Lie algebra, $\g L_X$ is generated by $H^0(\Om_X^1)^*$ (see Lemma \ref{01.07.2020--1}). Now we apply 
Theorem \ref{25.06.2020--3} to conclude that the image of ${\mathbf G}(\g L_X)$ in $\bb{GL}(E)$ is the 
group-envelope of the Lie algebra generated by $A(H^0(\Om_X^1)^*)$. Because of Proposition \ref{30.06.2020--2}, the 
image of $\Te(X,\,x_0)\,\simeq\, {\mathbf G}( \g L_X)$ is the image of $\Pi(X,\,x_0)$, which is the differential 
Galois group.
\end{proof}

\begin{rmk}
In ``birational'' differential Galois theory, one can find a result reminiscent of Theorem
\ref{25.06.2020--1}; see \cite[p.~25, Remarks 1.33]{van_der_put-singer03}.
\end{rmk}

\begin{rmk}One could hope for a straightforward way to compute differential Galois groups, as Theorem \ref{25.06.2020--1}, in the case where $X$ fails to be projective  and $(\ce,\na)$ is taken to be a {\it regular-singular connection}. But this is   certainly false: take $X\,=\,\spc K[x,\,x^{-1}]$
and define $(\co_{X}e,\,\na)$ by $\displaystyle\na e\,=\,ke\ot\fr{dx}{x}$ for any given $k\,\in\,\ZZ\smallsetminus\{0\}$. In this case, the differential Galois group is trivial (since $\na(x^{-k}e)=0$), while the Lie algebra generated by $k\in\ K$ is not.
\end{rmk}

Fixing generators of the Lie algebra of a subgroup scheme of some general linear group allows us to construct connections with a 
prescribed differential Galois group.
To state our results, we need the notion of semi-simple and reductive group schemes over a general field.   Let  $\ov K$ be an algebraic closure of  $K$. An algebraic group scheme  $G$ (over $K$) is {\it semi-simple}, respectively {\it reductive}, if and only if   $G\ot\ov K$ is a semi-simple, respectively  reductive, group scheme over $\ov K$ \cite[6.44, 6.46]{milne17}. Since in the case of reductive group schemes our arguments require a bit more group theory, we treat the semi-simple and reductive cases separately.

\begin{cor}\label{14.06.2021--1}
Let $X$ be a projective curve (smooth and integral, by assumption) over $K$ of genus $g$ and carrying a point $x_0\,\in\, X(K)$.
Let $G$ be a connected algebraic group scheme with Lie algebra $\g g$. 
\begin{enumerate}[(1)]\item Let 
\[
\mu=\min\left\{\dim V\,:\,\begin{array}{c}\text{$V$ is a subspace of $\g g$ which }\\ \text{  generates $\g g$ as a Lie algebra}
\end{array}\right\}.
\]Then, there exists a trivial vector bundle with a connection having differential Galois group $G$ if   $\mu\le g$. 
\item Suppose that $g\ge2$ and that   $G$ is  semi-simple. Then there exists a trivial vector bundle with a connection having differential Galois group $G$.
\item[(2bis)] In the setting of the previous item, if $G$ is, in addition, split, then the   connection  can be written down explicitly. (For the definition of ``split'', see \cite[Definition 19.22, p.402]{milne17}.)
\end{enumerate}
\end{cor}

\begin{proof}
(1) Let $V\subset\g g$ be a subspace of $\g g$ of dimension  $\mu$ which generates $\g g$ as a Lie algebra. 
  Let $G\,\longrightarrow\, \bb{GL}(E)$ be a closed 
immersion and regard $V$ as a subspace of  $\mm{End}(E)$. We then pick any $K$-linear map $A\,:\,H^0( \Om_X^1)^*\,\longrightarrow\, \mm{End}(E)$ such that ${\rm Im}(A)\,=\,V$. This map then becomes a morphism of $K$-algebras $A\,:\,\bb TH^0( \Om_X^1)^*\,\longrightarrow\, \mm{End}(E)$. 
Note that $G$ is   the 
group-envelope of $V$ and by Theorem \ref{25.06.2020--1}, the differential Galois group of $\cv(E,A)$ is isomorphic to $G$.

(2) We begin by recalling that a Lie algebra   over $K$ is semi-simple if and only its base change to   $\ov K$ is likewise   \cite[I.6.10]{bourbaki_algebres_lie}.  Since $G\ot \ov K$ is semi-simple if and only if ${\rm Lie}(G\ot\ov K)\simeq ({\rm Lie}\,G)\ot\ov K$ is semi-simple (see either \cite[Corollary II.6.2.2]{demazure-gabriel70} or \cite[Proposition 27.2.2]{tauvel-yu05}), we can assure that  $\g g\,:=\,{\rm Lie}(G)$ is semi-simple. According to Kuranishi's theorem (see 
\cite[Theorem 1]{kuranishi49} or \cite[VIII.2, p.~221, Exercise 8]{bourbaki_algebres_lie}), there exists a two 
dimensional vector space $V\,\subset\,\g g$ generating $\g g$ as a Lie algebra so the previous item can be applied.

(2bis) Let $T$ be a  maximal torus of $G$ which is split. Let us write   $\g t$ for $\mm{Lie}\,T$. Then, $\g g=\g t\op \bigoplus_{\al\in R}\g g_\al$, where $\g g_\al$ is the eigenspace associated to a non-trivial character $\al:T\longrightarrow\bb G_{m,K}$ \cite[21.a]{milne17}. For convenience, we shall also denote by $\al$ the differential $\g t\aro K$ obtained from $\al$. 
From the direct sum decomposition above, we see that   $\g t$ is a Cartan subalgebra of $\g g$. 
 Let $\eta_\al\in\g g_\al\setminus\{0\}$ and let 
\[
\xi\in \g t \setminus\bigcup_{\al\in R}\mm{Ker}(\al)\cup\bigcup_{\substack{\al,\be\in R\\\al\not=\be}}\mm{Ker}(\al-\be).
\]Then \[\xi\quad\text{and}\quad\eta:=\sum_{\al}\eta_\al\] are generators of $\g g$   \cite[VIII.2, p.~221, Exercise 8]{bourbaki_algebres_lie}. 
We can now proceed as in (1): Let $G\subset\bb{GL}(E)$ and let $\xi,\eta\in\g{g}$ be   interpreted as endomorphisms of $E$. Then, if $\ph,\ps\in H^0(\Om_X^1)$ are linearly independent,  define a connection on $\co\ot E$ by 
\[
\na (1\ot e)= (1\ot\xi(e))\ot \ph +(1\ot\eta(e))\ot\ps.\]
Its differential Galois group is isomorphic to $G$.
\end{proof}

Corollary \ref{14.06.2021--1}-(2) can be modified to throw light on the case of reductive groups. Some preparatory material is necessary.

Let $G$ be a reductive group scheme   with Lie algebra $\g g$. Then $G\ot \ov K$ is a reductive $\ov K$-group scheme and hence $\g g\ot\ov K=\mm{Lie}(G\ot\ov K)$ is  a reductive $\ov K$-Lie algebra \cite[Proposition 27.2.2]{tauvel-yu05}. By \cite[I.6.10]{bourbaki_algebres_lie}
$\g g$ is reductive. Hence, $\g g=\g s\op\g z$, where $\g s$ is semi-simple and $\g z$ is the center of $\g g$ \cite[I.6.4, Proposition 5]{bourbaki_algebres_lie}. As is well-known, if $Z$ stands for the {\it neutral component of the center} of $G$, then $\mm{Lie}\,Z=\g z$    \cite[II.6, Proposition 2.1, p.259]{demazure-gabriel70}. Moreover,  $Z$ is a geometrically connected group scheme of multiplicative type  (see Proposition 1.34 and Corollary 17.62 in \cite {milne17}), and hence a {\it torus}.

\begin{lem}Suppose that $[K:\mathbf Q]=\infty$. Then, there exists a subset  $\{\xi,\eta,\zeta\}\subset\g g$ with the following property. The Lie algebra    generated by $\{\xi,\eta,\zeta\}$ is algebraically dense in $\g g$. Said differently, any closed subgroup scheme $H\subset G$ such that $\{\xi,\eta,\zeta\}\subset\mm{Lie}\,H$ must actually coincide with $G$. 
\end{lem}
\begin{proof}Let $\xi,\eta\in\g s$ generate $\g s$ as a Lie algebra. 
We now show that there exists $\zeta\in\g z$ such that $K\zeta$ is algebraically dense in $\g z$. 
Since $Z$ is a torus,  
the $\ov K$-group scheme  $Z\ot \ov K$ is isomorphic to   $\mathbf G_{m,\ov K}^r$. From \cite[24.6.3]{tauvel-yu05} and the hypothesis that $[K:\mathbf Q]=\infty$, there exists $\zeta\in\g z$ such that $\ov K\po(\zeta\ot1)$  is algebraically dense in $\mm{Lie}(Z\ot \ov K)$: it is enough to pick $\ze=(\ze_1,\ldots,\ze_r)\in K^r$ with $\{\ze_i\}$ linearly independent over $\mathbf Q$.  It is then a simple matter to see that $K\zeta\subset\g z$ is algebraically dense.

To end, let $H\subset G$ be a closed subgroup scheme whose Lie algebra $\g h$  contains $\{\xi,\eta,\zeta\}$. Then $\ze\in\g h\cap\g z= \mm{Lie}(H\cap Z)$ \cite[10.14]{milne17} and hence $\g z\subset \g h$, which shows, using the equality $\g g=\g s+\g z$, that $\g h=\g g$ and we conclude by \cite[II.6, Proposition 2.1]{demazure-gabriel70}. 
\end{proof}

The same method used to establish Corollary \ref{14.06.2021--1} now gives the following. 
\begin{cor}\label{04.01.2023--1}Let $X$,   $x_0$ and $g$ be as in Corollary \ref{14.06.2021--1}. Let $G$ be a reductive group scheme. If  $[K:\mathbf Q]=\infty$ and $g\ge3$, then $G$ is the differential Galois group of a connection on a trivial vector bundle. In addition, if $G$ is split, the construction can be made explicit. \qed
\end{cor}

Let us illustrate how explicit the constructions can be made with an example. All depends on the construction of generating elements in a Lie algebra. 

\begin{ex}Let $K=\QQ$ and  $X$ be an arbitrary smooth and projective curve carrying a $K$-rational point  and having genus at least 2.  We give ourselves   non-proportional global differential forms $\ph$ and $\ps$. We wish to construct an explicit expression of a connection on a free vector bundle on $X$ such that the associated differential Galois group is an ``exceptional'' group scheme. To explain what exceptional means requires some preliminary material from \cite{bourbaki_algebres_lie} and \cite{milne17}.

Recall from the ``Isogeny Theorem''   that there exists a semi-simple and split  group scheme $G$ whose root system is $G_2$ and, in addition, such group is unique up to isomorphism. This is clearly proved as  \cite[Theorem 23.25, pp 492-3]{milne17}, albeit in terms of root data. The (well-known) link between root data  and  root systems  is explained by the concept of semi-simple root data \cite[Definition C.34, p.615]{milne17} and ``diagrams'' 
\cite[Definition C.27, p.613]{milne17}. 
See also \cite[Theorem 23.58, p.501]{milne17} for a concise statement. 

We set out to construct a connection on $X$ whose differential Galois group is isomorphic to $G$. To employ our method, we need to realise $\mm{Lie}\,G$ explicitly as an algebra of matrices and then spot generators; for the first task  we follow  \cite{hesselink19} because of the explicitness of its  constructions over $\mathbf Q$. 
Let $V$ be an {\it eight}-dimensional vector space with basis $\{b_i\}_{i=0}^7$. For $0\le i,j\le 7$, define     $e_{ij}\in\mm{End}(V)$  by 
\[
b_k\longmapsto\left\{\begin{array}{ll}
0,&\text{if $k\not=j$,}
\\
b_i,&\text{if $k=j$.}
\end{array}\right. 
\]
We then define  ``upper-triangular maps'' 
\[\begin{array}{lllllllll}
x_0=  e_{01}+e_{23}-e_{24}+e_{ 35}-e_{45}-e_{67},&&& x_1= e_{12}-e_{56},&&&
\\ x_2=-e_{02}+e_{13}-e_{14}+e_{36}-e_{46}+e_{57},
\\
x_3=e_{03}-e_{04}-e_{15}+e_{26}+e_{37}-e_{47},&&&x_4=-e_{05}+e_{27},&&&x_5=-e_{06}+e_{17},
\end{array}\]
 ``diagonal maps'' 
\[
\begin{array}{llll}x_6=e_{00}+e_{11}-e_{66}-e_{77},&&&x_7= e_{00}+e_{22}-e_{55}-e_{77},\end{array}
\]
and ``lower-triangular maps''
\[
\begin{array}{llllllll}
x_8=e_{60}-e_{71},&&&x_9=e_{50}-e_{72},&&&\\
x_{10} =-e_{30}+e_{40}+e_{51}-e_{62}-e_{73}+e_{74},
\\
x_{11}=e_{20}-e_{31}+e_{41}-e_{63}+e_{64}-e_{75},&&&x_{12}=-e_{21}+e_{65},&&&
\\x_{13}=-e_{10}-e_{32}+e_{42}-e_{53}+e_{54}+e_{76}. 
\end{array}
\]
(See p. 631  and section 3.3 in  \cite{hesselink19}, but beware that he uses a capital $X$ to denote endomorphisms.) By computer-algebra, one can verify that the span of $\{x_i\}_{i=0}^{13}$
is a Lie algebra $\g G$ of dimension 14 and   that  $\g T:=Kx_6+Kx_7$ is an abelian subalgebra. 
If $\{x_6^*,x_7^*\}$ is the dual basis of $\{x_6,x_7\}$, then define  
\[\begin{array}{lllll}\al_0=x_7^*,&&\al_1=x_6^*-x_7^*,&&\al_2=x_6^*,
\\
\al_3=x_6^*+x_7^*,&&\al_4=x_6^*+2x_7^*,&&\al_5=2x_6^*+x_7^*.
\end{array}
\]
By computer-algebra, 
we verify that for  each $i\in \{0,\ldots,5\}$,  we have 
\[
\mm{ad}_{t}(x_i)=\al_i(t)x_i,\qquad\forall\,t\in \g T
\]  and 
\[
\mm{ad}_t(x_{13-i}) = -\al_i(t)x_{13-i},\qquad\forall\,t\in\g T.
\]
Consequently,
 \[
\g G=\g T\op \bigoplus_{i=0}^5Kx_i\op Kx_{13-i}
\]
is the root decomposition of $(\g G,\g T)$ and defines a root system of type $G_2$ on the vector space $\g T^*=Kx_6^*+Kx_7^*$, see  \cite[Figure 2, p.634]{hesselink19}.

Following the method described in Corollary \ref{14.06.2021--1}-(3), we know that the connection  on $V\ot\co_X$ determined by the $\mm{End}(V)$-valued form
\[
\underbrace{
\left(\begin{array}{rrrrrrrr}
0 & 1 & -1 & 1 & -1 & -1 & -1 & 0 \\
-1 & 0 & 1 & 1 & -1 & -1 & 0 & 1 \\
1 & -1 & 0 & 1 & -1 & 0 & 1 & 1 \\
-1 & -1 & -1 & 0 & 0 & 1 & 1 & 1 \\
1 & 1 & 1 & 0 & 0 & -1 & -1 & -1 \\
1 & 1 & 0 & -1 & 1 & 0 & -1 & 1 \\
1 & 0 & -1 & -1 & 1 & 1 & 0 & -1 \\
0 & -1 & -1 & -1 & 1 & -1 & 1 & 0
\end{array}\right)
}_{\sum_{i=0}^5x_i+x_{13-i}}\ot\ph+\underbrace{\left(\begin{array}{rrrrrrrr}
4 & 0 & 0 & 0 & 0 & 0 & 0 & 0 \\
0 & 1 & 0 & 0 & 0 & 0 & 0 & 0 \\
0 & 0 & 3 & 0 & 0 & 0 & 0 & 0 \\
0 & 0 & 0 & 0 & 0 & 0 & 0 & 0 \\
0 & 0 & 0 & 0 & 0 & 0 & 0 & 0 \\
0 & 0 & 0 & 0 & 0 & -3 & 0 & 0 \\
0 & 0 & 0 & 0 & 0 & 0 & -1 & 0 \\
0 & 0 & 0 & 0 & 0 & 0 & 0 & -4
\end{array}\right)}_{x_6+3x_7}\ot \ps\]
has    differential Galois group isomorphic to $G$. Indeed, as $\g G$ is semi-simple, it must be the Lie algebra of a closed and connected subgroup scheme of $\bb{GL}(V)$ \cite[Proposition II.6.2.6, p.262]{demazure-gabriel70}, which must be isomorphic to $G$. 
\end{ex}

\end{document}